\documentclass[12pt]{amsart}

\usepackage{amsmath,amssymb,amsthm,here}
\usepackage[all]{xypic}
\usepackage{here} 
\usepackage{color} 

\numberwithin{equation}{section}

\SelectTips{eu}{12}

\newtheorem{theorem}{Theorem}[section]
\newtheorem{proposition}[theorem]{Proposition}
\newtheorem{lemma}[theorem]{Lemma}
\newtheorem{corollary}[theorem]{Corollary}

\theoremstyle{definition}

\theoremstyle{remark}
\newtheorem{remark}[theorem]{Remark}

\newcommand{\Z}{\mathbb{Z}}

\newcommand{\U}{\mathrm{U}}
\newcommand{\SU}{\mathrm{SU}}
\newcommand{\Sp}{\mathrm{Sp}}
\newcommand{\Spin}{\mathrm{Spin}}
\newcommand{\E}{\mathrm{E}}
\newcommand{\F}{\mathrm{F}}
\newcommand{\G}{\mathrm{G}}
\newcommand{\g}{\mathcal{G}}
\newcommand{\ch}{\mathrm{ch}}
\newcommand{\map}{\mathrm{map}}

\title{Odd primary homotopy types of the gauge groups of exceptional Lie groups}

\author{Sho Hasui} 
\address{Faculty of Liberal Arts and Sciences, Osaka Prefecture University, Sakai, 599-8531, Japan}
\email{s.hasui@las.osakafu-u.ac.jp} 

\author{Daisuke Kishimoto}
\address{Department of Mathematics, Kyoto University, Kyoto, 606-8502, Japan}
\email{kishi@math.kyoto-u.ac.jp}

\author{Tseleung So}
\address{Mathematical Sciences, University of Southampton, Southampton SO17 1BJ, United Kingdom}
\email{tls1g14@soton.ac.uk} 

\author{Stephen Theriault}
\address{Mathematical Sciences, University of Southampton, Southampton SO17 1BJ, United Kingdom}
\email{s.d.theriault@soton.ac.uk} 


\begin{document}

\baselineskip 17pt

\maketitle

\begin{abstract}
The $p$-local homotopy types of gauge groups of principal $G$-bundles over $S^4$ are classified when $G$ is a compact connected exceptional Lie group without $p$-torsion in homology except for $(G,p)=(\E_7,5)$.
\end{abstract}


\section{Introduction}

Let $G$ be  a topological group and $P$ be a principal $G$-bundle over a base $X$. The \emph{gauge group} $\g(P)$ is the topological group of $G$-equivariant self-maps of $P$ covering the identity map of $X$. As $P$ ranges over all principal $G$-bundles over $X$, we get a collection of gauge groups $\g(P)$. In \cite{K}, Kono classified the homotopy types in the collection of $\g(P)$ when $G=\SU(2)$ and $X=S^4$. Since then, the homotopy theory of gauge groups has been deeply developed in connection with mapping spaces and fiberwise homotopy theory incorporating higher homotopy structures. The classification of the homotopy types has been its motivation and considerable progress has been made towards it.

Suppose that $G$ is a compact connected simple Lie group. Then there is a one-to-one correspondence between principal $G$-bundles over $S^4$ and $\pi_3(G)\cong\Z$. Let $\g_k$ denote the gauge group of the principal $G$-bundle over $S^4$ corresponding to $k\in\Z$. Most of the classification of the homotopy types of gauge groups has been done in this setting, generalizing the situation studied by Kono \cite{C,HK,HKKS,KKKT,KTT,T1,T2}. As we will see, the homotopy type of $\g_k$ is closely related with a certain Samelson product in $G$. It is shown in \cite{KKT} that if $p$ is a prime large with respect to the rank of $G$, then the classification of the $p$-local homotopy types of $\g_k$ reduces to determining the order of this Samelson product. So in \cite{KK,T3}, the $p$-local classification is obtained by calculating this Samelson product when $G$ is a classical group.

However, when $G$ is exceptional, the classification, up to one factor of 2,  for $G=\G_2$ \cite{KTT} is the only one result obtained so far. The aim of this paper is to classify the $p$-local homotopy types of $\g_k$ when $G$ is exceptional and has no $p$-torsion in $H_*(G;\Z)$. Recall that exceptional Lie groups consist of five types $\G_2,\F_4,\E_6,\E_7,\E_8$ and they have no $p$-torsion in $H_*(G;\Z)$ if and only if
\begin{equation}
\label{torsion_free}
\G_2,\;p\ge 3;\quad\F_4,\E_6,\E_7,\;p\ge 5;\quad\E_8,\;p\ge 7.
\end{equation}
All but one case is included in our results. We define an integer $\gamma(G)$ by:
\begin{alignat*}{3}
\gamma(\G_2)&=3\cdot 7&\gamma(\F_4)&=5^2\cdot 13&\gamma(\E_6)&=5^2\cdot 7\cdot 13\\
\gamma(\E_7)&=7\cdot 11\cdot 19\qquad&\gamma(\E_8)&=7^2\cdot 11^2\cdot 13\cdot 19\cdot 31\qquad
\end{alignat*}
For integers $m$ and $n$, let $(m,n)$ denote the greatest common divisor of $m$ and $n$, and for an integer $k=p^rq$ with $(p,q)=1$, let $\nu_p(k)=p^r$ denote the $p$-component of $k$. Now we state our main result.

\begin{theorem}
\label{main}
Let $G$ be a compact connected exceptional simple Lie group without $p$-torsion in $H_*(G;\Z)$ except for $(G,p)=(\E_7,5)$. Then $\g_k$ and $\g_l$ are $p$-locally homotopy equivalent if and only if $\nu_p((k,\gamma(G)))=\nu_p((l,\gamma(G)))$.
\end{theorem}


\section{Gauge groups and Samelson products}

We first recall a relation between gauge groups and Samelson products. Let $G$ be a compact connected simple Lie group and let $\epsilon$ be a generator of $\pi_3(G)\cong\Z$. Then by definition $\g_k$ is the gauge group of a principal $G$-bundle over $S^4$ corresponding to $k\epsilon$. Let $\map(X,Y;f)$ denote the connected component of the mapping space $\map(X, Y)$ containing the map $f\colon X\to Y$. It is proved by Gottlieb \cite{G} (cf. \cite{AB}) that there is a homotopy equivalence
$$B\g_k\simeq\map(X,BG;k\bar{\epsilon}),$$
where $\bar{\epsilon}\colon S^4\to BG$ is the adjoint of $\epsilon\colon S^3\to G$. Then by evaluating at the basepoint of $S^4$, we get a homotopy fibration sequence
$$\g_k\to G\xrightarrow{\partial_k}\Omega^3_0G\to B\g_k\to BG,$$
so $\g_k$ is homotopy equivalent to the homotopy fiber of the map $\partial_k$. Thus we must identify the map $\partial_k$, which was done by Lang \cite{L}.

\begin{lemma}
\label{connecting}
The adjoint $S^3\wedge G\to G$ of the map $\partial_k$ is the Samelson product $\langle k\epsilon,1_G\rangle$.
\end{lemma}

By the linearity of Samelson products, we have:

\begin{corollary}
$\partial_k\simeq k\circ\partial_1$.
\end{corollary}

Hereafter we localize spaces and maps at a prime $p$. We next recall the result of Kono, Tsutaya and the second author \cite{KKT} on the $p$-local homotopy types of $\g_k$. An H-space $X$ is called \emph{rectractile} if the following conditions are satisfied:

\begin{enumerate}
\item there is a space $A$ such that $H_*(X)=\Lambda(\widetilde{H}_*(A))$;
\item there is a map $\lambda\colon A\to X$ which is the inclusion of the generating set in homology.
\item $\Sigma\lambda$ has a left homotopy inverse.
\end{enumerate}

The space $A$ is often called a homology generating subspace. By \cite{T0} it is known that a compact connected simple Lie group $G$ is rectractile at the prime $p$ if $(G,p)$ is in the following table.

\renewcommand{\arraystretch}{1.2}
\begin{table}[htbp] 
\centering
\begin{tabular}{llllll} 
$\SU(n)$&$(p-1)^2+1\ge n$&$\G_2$&$p\ge 3$\\
$\Sp(n),\Spin(2n+1)$&$(p-1)^2+1\ge 2n$&$\F_4,\E_6$&$p\ge 5$\\
$\Spin(2n)$&$(p-1)^2+1\ge 2(n-1)$&$\E_7,\E_8$&$p\ge 7$
\end{tabular}
\end{table}

Note that by \eqref{torsion_free}, there is no $p$-torsion in $H_*(G;\Z)$ if and only if $G$ is rectractile at the prime~$p$, except for $(G,p)=(\G_2,3),(\E_7,5)$. Let $\gamma(G,p)$ be the order of the Samelson product $\langle\epsilon,\lambda\rangle$. In \cite{KKT} the following reduction of the classification of the $p$-local homotopy types of $\g_k$ is proved.

\begin{theorem}
\label{KKT}
Let $G$ be a compact connected simple Lie group which is rectractile at the prime~$p$ except for 
$(G,p)=(\G_2,3)$. Then $\g_k$ and $\g_l$ are $p$-locally homotopy equivalent if and only if $(k,\gamma(G,p))=(l,\gamma(G,p))$.
\end{theorem}

\begin{remark}
The range of $p$ in \cite{KKT} is smaller than the above statement because the range in \cite{KKT} was taken so that $G$ is rectractile and the mod $p$ decomposition satisfies a certain universality. However, we can easily see that only the retractile property was used in \cite{KKT}, so we can improve the range of $p$ as above. 
\end{remark}

We set notation for $G$. By the classical result of Hopf, any connected finite H-space is rationally homotopy equivalent to the product of odd dimensional spheres $S^{2n_1-1}\times\cdots\times S^{2n_l-1}$ for $n_1\le\cdots\le n_l$. The sequence $(n_1,\cdots, n_l)$ is called the type of $X$. Hereafter, we assume that $G$ is of type $(n_1,\cdots, n_l)$, without $p$-torsion, and rectractile at $p$.

We next decompose the Samelson product $\langle\epsilon,\lambda\rangle$ in $G$ into small pieces. Recall that since we are localizing at $p$, $G$ decomposes as
$$G\simeq B_1\times\cdots\times B_{p-1}$$
where the type of $B_i$ consists of the integers $n$ in the type of $G$ with $2n-1$ congruent to $2i+1$ modulo $2(p-1)$ \cite{MNT}. This is called the mod $p$ decomposition of $G$. If the rank of $B_i$ is less than~$p$, then there is a subspace $A_i$ of $B_i$ such that $B_i$ is rectractile with respect to the inclusion $A_i\to B_i$. From this, one can deduce the retractile property for $G$ as above, so we may take the space $A$ as $A=A_1\vee\cdots\vee A_{p-1}$. Then it follows that the Samelson product $\langle\epsilon,\lambda\rangle$ is the direct sum of $\langle\epsilon,\lambda_1\rangle,\ldots,\langle\epsilon,\lambda_{p-1}\rangle$, where $\lambda_i\colon A_i\to G$ is the restriction of $\lambda$ to $A_i$. Let $\gamma_i(G,p)$ be the order of $\langle\epsilon,\lambda_i\rangle$. Then we have:

\begin{proposition}
\label{max}
If $G$ is rectractile at the prime $p$, then 
$$\gamma(G,p)=\max\{\gamma_1(G,p),\ldots,\gamma_{p-1}(G,p)\}.$$
\end{proposition} 

Thus if we calculate $\gamma_{i}(G,p)$ for $1\leq i\leq p-1$ for a given pair $(G,p)$ 
then by Proposition~\ref{max} we know $\gamma(G,p)$ and by Theorem~\ref{KKT} we can 
describe the $p$-local homotopy types of the corresponding gauge groups. These calculations 
are carried out in the subsequent sections.

The following table lists explicit mod $p$ decompositions of exceptional Lie groups without $p$-torsion in homology except for $\G_2$, where $B(2m_1-1,\ldots,2m_k-1)$ is an H-space of type $(m_1,\ldots,m_k)$. One can easily deduce the types of exceptional Lie groups from this table.

\renewcommand{\arraystretch}{1.2}
\begin{table}[H] 
\centering
\begin{tabular}{lll} 
$\F_4$&$p=5$&$B(3,11)\times B(15,23)$\\
&$p=7$&$B(3,15)\times B(11,23)$\\
&$p=11$&$B(3,23)\times S^{11}\times S^{15}$\\ 
&$p\geq 13$ & $S^{3}\times S^{11}\times S^{15}\times S^{23}$\\ 
$\E_6$&$p=5$&$\F_4\times B(9,17)$\\
&$p\ge 7$&$\F_4\times S^9\times S^{17}$\\
$\E_7$&$p=5$&$B(3,11,19,27,35)\times B(15,23)$\\
&$p=7$&$B(3,15,27)\times B(11,23,35)\times S^{19}$\\
&$p=11$&$B(3,23)\times B(15,35)\times S^{11}\times S^{19}\times S^{27}$\\
&$p=13$&$B(3,27)\times B(11,35)\times S^{15}\times S^{19}\times S^{23}$\\
&$p=17$&$B(3,35)\times S^{11}\times S^{15}\times S^{19}\times S^{23}\times S^{27}$\\ 
&$p\geq 19$&$S^{3}\times S^{11}\times S^{15}\times S^{19}\times S^{23}\times S^{27}\times S^{35}$\\
$\E_8$&$p=7$&$B(3,15,27,39)\times B(23,35,47,59)$\\
&$p=11$&$B(3,23)\times B(15,35)\times B(27,47)\times B(39,59)$\\
&$p=13$&$B(3,27)\times B(15,39)\times B(23,47)\times B(35,59)$\\
&$p=17$&$B(3,35)\times B(15,47)\times B(27,59)\times S^{23}\times S^{39}$\\
&$p=19$&$B(3,39)\times B(23,59)\times S^{15}\times S^{27}\times S^{35}\times S^{47}$\\
&$p=23$&$B(3,47)\times B(15,59)\times S^{23}\times S^{27}\times S^{35}\times S^{39}$\\
&$p=29$&$B(3,59)\times S^{15}\times S^{23}\times S^{27}\times S^{35}\times S^{39}\times S^{47}$\\ 
&$p\geq 31$&$S^{3}\times S^{15}\times S^{23}\times S^{27}\times S^{35}\times S^{39}\times S^{47}\times S^{59}$
\end{tabular}
\end{table}

If each $B_i$ is a sphere, then we call $G$ \emph{$p$-regular}. If each $B_i$ is a sphere or a rank 2 H-space $B(2m-1,2m+2p-3)$, then we call $G$ \emph{quasi-$p$-regular}. In the following we calculate $\gamma(G,p)$ for different cases.


\section{The $p$-regular case}

Throughout this section, we let $(G,p)$ be
\begin{equation}
\label{regular}
\F_4,\E_6,\;p\ge 13;\quad\E_7,\;p\ge 19;\quad\E_8,\;p\ge 31.
\end{equation}
By the above table, this is equivalent to the case when $G$ is $p$-locally homotopy equivalent to a product of spheres, and such a $G$ is called $p$-regular. This is also equivalent to $p\ge n_l+1$.

The homotopy groups of $p$-localized spheres are known for low dimension.
\begin{theorem}[Toda \cite{To}]\label{homotopy group of sphere}
Localized at $p$, we have
\[
\pi_{2n-1+k}(S^{2n-1})\cong\begin{cases}
\Z/p\Z	&\text{for $k=2i(p-1)-1$, $1\leq i\leq p-1$}\\
\Z/p\Z	&\text{for $k=2i(p-1)-2$, $n\leq i\leq p-1$}\\
0		&\text{other cases for $1\leq k\leq 2p(p-1)-3$},
\end{cases}
\]
\end{theorem}

By Theorem \ref{homotopy group of sphere} we have $p\cdot\pi_{2k}(S^{2n-1})=0$ for $0\le k\le n+p(p-1)-2$, so $p\cdot\pi_{2k}(G)=0$ for $0\le k\le p(p-1)$. Then since the Samelson product $\langle\epsilon,\lambda_i\rangle$ is a map from a $2n$-dimensional sphere into $G$ for $n\le n_l+1\le p$, we get:

\begin{proposition}
\label{upper_bound_regular}
If $G$ is $p$-regular, then $\gamma_i(G,p)\le p$ for any $i$.
\end{proposition}

On the other hand, the non-triviality of the Samelson product $\langle \lambda_i,\lambda_j\rangle$ is completely determined in \cite{HKO}, where $\lambda_i$ is the inclusion $S^{2n_i-1}\to G$ of the sphere factor in the mod $p$ decomposition of $G$.

\begin{theorem}
The Samelson product $\langle \lambda_i,\lambda_j\rangle$ is non-trivial if and only if $n_i+n_j=n_k+p-1$ for some $n_k$ in the type of $G$.
\end{theorem}

Then we immediately get:

\begin{corollary}
\label{triviality_regular}
\begin{enumerate}
\item The Samelson product $\langle\epsilon,\lambda_l\rangle$ is non-trivial for $p=n_l+1$;
\item The Samelson product $\langle\epsilon,\lambda_i\rangle$ is trivial for any $i$ for $p>n_l+1$.
\end{enumerate}
\end{corollary}

Thus by combining Propositions \ref{max} and \ref{upper_bound_regular} and Corollary \ref{triviality_regular}, we obtain the following.

\begin{corollary}
\label{cor:regular}
We have
$$\gamma(G,p)=\begin{cases}p&p=n_l+1\\1&p>n_l+1.\end{cases}$$
\end{corollary}


\section{The quasi-$p$-regular case}

Throughout this section, we assume that $(G,p)$ is
\begin{equation}
\label{regular}
\F_4,\E_6,\;p=5,7,11;\quad\E_7,\;p=11,13,17;\quad\E_8,\;p=11,13,17,19,23,29
\end{equation}
unless otherwise stated. This is equivalent to $G$ being a product of spheres and rank 2 H-spaces, and such a $G$ is quasi-$p$-regular. 

We first show the triviality of $\langle\epsilon,\lambda_i\rangle$ from the homotopy groups of $G$. By Theorem \ref{homotopy group of sphere}, for $n\ge 1$ and $i\le n+(p-1)^2$, we have $\pi_{2i}(S^{2n+1})=0$  unless $i\equiv n\mod(p-1)$ and $p\cdot\pi_{2i}(S^{2n+1})=0$.

Moreover, the homotopy groups of the H-spaces $B(2n-1,2n-1+2(p-1))$ in the mod $p$ decomposition of $G$ are calculated.

\begin{theorem}[Mimura and Toda \cite{MT}, Kishimoto \cite{Ki}]\label{homotopy group of Bi}
Localized at $p$, we have
\[
\pi_{2n-1+k}(B(3, 2p+1))\cong\begin{cases}
\Z/p\Z	&\text{for $k=2i(p-1)-1$, $2\leq i\leq p-1$}\\
\Z		&\text{for $k=2p-2$}\\
0		&\text{other cases for $1\leq k\leq 2p(p-1)-3$},
\end{cases}
\]
and
\[
\pi_{2n-1+k}(B(2n-1, 2n+2p-3))\cong\begin{cases}
\Z/p^2\Z	&\text{for $k=2i(p-1)-1$, $2\leq i\leq p-1$}\\
\Z/p\Z		&\text{for $k=2i(p-1)-2$, $n\leq i\leq p-1$}\\
\Z			&\text{for $k=2p-2$}\\
0			&\text{other cases for $1\leq k\leq 2p(p-1)-3$}.
\end{cases}
\]
\end{theorem}

For $n\ge 1$ and $i\le n+(p-1)^2$, $\pi_{2i}(B(2n+1,2n+1+2(p-1)))=0$ unless $i\equiv n\mod(p-1)$, and $p\cdot\pi_{2i}(B(3,2p+1))=0$ for $i\le 1+(p-1)^2$. By Theorems \ref{homotopy group of sphere} and \ref{homotopy group of Bi}, we obtain the following table for all possibly non-trivial Samelson products $\langle\epsilon,\lambda_i\rangle$ in $G$:

\renewcommand{\arraystretch}{1.2}
\begin{table}[H] 
\centering
\begin{tabular}{l|lllll} 
&$p=5$&$p=7$&$p=11$&$p=13$&p=19\\\hline
$\F_4$&$i=3$&$i=5$\\
$\E_6$&$i=3$&$i=2,5$\\
$\E_7$&&&$i=3,7$&$i=5$\\
$\E_8$&&&$i=9$&$i=5,11$&$i=5,11$
\end{tabular}
\end{table}

We will determine the order of these Samelson products by using the technique developed in \cite{HKMO} which we briefly recall here. The technique in \cite{HKMO} is for computing the Samelson products in the $p$-localized exceptional Lie group $G$ for $(G,p)$ in \eqref{torsion_free} except for $(G,p)=(\G_2,3),(\E_7,5)$. We only consider the case that $G=\E_6,\E_7,\E_8$ since the $\F_4$ case can be deduced from the $\E_6$ case. Consider the 27-, 56- and 248-dimensional irreducible representations of $G=\E_6,\E_7,\E_8$, respectively. Let $\rho\colon G\to\SU(\infty)$ be the composite of these representations and the inclusions into $\SU(\infty)$. Then we have a homotopy fibration sequence
$$\Omega\SU(\infty)\xrightarrow{\delta}F\to G\xrightarrow{\rho}\SU(\infty)$$
in which all maps are loop maps, where $F$ is the homotopy fiber of the map $\rho$. For any finite CW complex $Z$, we obtain an exact sequence of groups
$$\widetilde{K}^0(Z)\cong\widetilde{K}^{-2}(Z)\xrightarrow{\delta_*}[Z,F]\to[Z,G]\to\widetilde{K}^{-1}(Z).$$
In particular, if $Z$ has only even dimensional cells, then $\widetilde{K}^{-1}(Z)=0$ and we get
\begin{equation}
\label{[Z,G]}
[Z,G]\cong[Z,F]/\mathrm{Im}\,\delta_*.
\end{equation}
In \cite{HKMO}, the right hand side of~(\ref{[Z,G]}) is identified by using the cohomology of $Z$ under some conditions and we recall it in the special case $Z=S^3\wedge A_i$. Let $\SU(\infty)\simeq C_1\times\cdots\times C_{p-1}$ be the mod $p$ decomposition of $\SU(\infty)$ where $C_i$ is of type $(i+1,i+p,\cdots,i+1+k(p-1),\cdots)$. 

\begin{theorem}
\label{HKMO}
Let $(G,p)$ be as in \eqref{torsion_free} except for $G=\G_2$ and $(G,p)=(\E_7,5)$. For given $1\le i\le p-1$, we suppose the following conditions:
\begin{enumerate}
\item for $1\le j\le p-1$ with $j\equiv i+2\mod(p-1)$, 
$$B_j=B(2j+1,2j+1+2(p-1),\ldots,2j+1+2r(p-1));$$
\item the composite $B_j\xrightarrow{\rm incl}G\xrightarrow{\rho}\SU(\infty)\xrightarrow{\rm proj}C_j$ is surjective in cohomology;
\item $\dim A_i\le 2j-3+2(r+2)(p-1)$.
\end{enumerate}
Then for $n=j+(r+1)(p-1)$, there is an injective homomorphism 
$$\Phi\colon[S^3\wedge A_i,F]\to H^{2n}(S^3\wedge A_i;\Z_{(p)})\oplus H^{2n+2(p-1)}(S^3\wedge A_i;\Z_{(p)})$$
such that
$$\Phi(\delta_*(\xi))=(n!\ch_n(\xi),(n+p-1)!\ch_{n+p-1}(\xi))$$
for $\xi\in\widetilde{K}(S^3\wedge A_i)$, where $\ch_k$ is the $2k$-dimensional part of the Chern character.
\end{theorem}

\begin{proof} 
To better understand the statement of the theorem, we sketch the proof in~\cite{HKMO}. 
Let~$\rho_j$ be the composition $B_j\overset{incl}{\longrightarrow}G\overset{\rho}{\longrightarrow}SU(\infty)\overset{proj}{\longrightarrow}C_j$ and let $F_j$ be its homotopy fiber. By hypothesis, $\tilde{H}^{i}(F_j;\Z_{(p)})$ is zero for $i\leq 2j+2(r+2)(p-1)$, except for $\tilde{H}^{2n}(F_j;\Z_{(p)})$ and $\tilde{H}^{2n+2(p-1)}(F_j;\Z_{(p)})$ which are $\Z_{(p)}$. Moreover, the generators $a_{2n}$ and $a_{2n+2(p-1)}$ of $\tilde{H}^{2n}(F_j;\Z_{(p)})$ and $\tilde{H}^{2n+2(p-1)}(F_j;\Z_{(p)})$ transgress to the suspensions of Chern class $c_{n+1}$ and $c_{n+p-1}$ modulo some decomposables respectively. Define 
\[\Phi:[S^3\wedge A_i, F_j]\to H^{2n}(S^3\wedge A_i;\Z_{(p)})\oplus H^{2n+2(p-1)}(S^3\wedge A_i;\Z_{(p)})\] 
by
\[
\Phi(\xi)=\xi^*(a_{2n})\oplus\xi^*(a_{2n+2(p-1)}).
\]
Then calculating as in \cite{HK} implies the asserted statement.
\end{proof}

This theorem can be applied to calculate Samelson products in $G$. Let $Z_1$ and $Z_2$ be spaces and let $\theta_1:Z_1\to G$ and $\theta_2:Z_2\to G$ be maps. Consider the Samelson product $\langle\theta_1,\theta_2\rangle$ and put $Z=Z_1\wedge Z_2$. Since $\SU(\infty)$ is homotopy commutative, $\langle\theta_1,\theta_2\rangle$ lifts to a map $\tilde{\theta}\colon Z\to F$. Then if $Z$ is a CW-complex consisting only of even dimensional cells, the order of $\tilde{\theta}$ in the coset $[Z,F]/\mathrm{Im}\,\delta_*$ is equal to the order of $\langle\theta_1,\theta_2\rangle$. In \cite{HKMO}, a convenient lift $\widetilde{\theta}$ is constructed such that, if we have the injection $\Phi$ of Theorem \ref{HKMO}, then $\Phi(\widetilde{\theta})$ is calculated from the Chern classes of $\rho$.

We now apply Theorem \ref{HKMO} to our case.

\begin{proposition}
\label{quasi-regular}
For $(G,p,i)=(\E_6,5,3),(\E_8,11,9)$, there is an injection 
$$\Phi\colon[S^3\wedge A,F]\to H^{2i+4+2(p-1)}(S^3\wedge A;\Z_{(p)})\oplus H^{2i+4+4(p-1)}(S^3\wedge A;\Z_{(p)})\cong(\Z_{(p)})^2$$
satisfying the following, where $A=A(2i+1+2(p-1),2i+1+4(p-1))$:
\begin{enumerate}
\item for $(G,p,i)=(\E_6,5,3)$, $\mathrm{Im}\,\Phi\circ\delta_*$ is generated by $(3\cdot 5,2^{-3}\cdot 5\cdot 13)$ and $(0,5^2)$, and a lift $\tilde{\theta}$ of $\langle\epsilon,\lambda_3\rangle$ can be chosen so that $\Phi(\tilde{\theta})=(2^{-3}\cdot 3^2,-2^{-6}\cdot 3^{-1}\cdot 71)$;
\item for $(G,p,i)=(\E_8,11,9)$, is generated by $(3\cdot 11\cdot 17\cdot 41,5^4\cdot 11\cdot 13\cdot 31\cdot 61)$ and $(0,11^2)$, and a lift $\tilde{\theta}$ of $\langle\epsilon,\lambda_3\rangle$ can be chosen so that $\Phi(\tilde{\theta})=(3\cdot 7\cdot 17\cdot 89,2^{-2}\cdot 3\cdot 5\cdot 2207\cdot 17977)$.
\end{enumerate}
\end{proposition}

\begin{proof}
The Chern classes of $\rho$ are calculated in \cite{HKMO}, so we can easily check that the conditions of Theorem \ref{HKMO} are satisfied in our case. We consider the case $(G,p,i)=(\E_6,5,3)$. Let $\bar{\rho}\colon\Sigma A\to B\SU(\infty)$ be the adjoint of $\rho\circ\lambda_i$ and put $\xi=\beta\wedge\bar{\rho}\in\widetilde{K}(S^3\wedge A)$, where $\beta$ is a generator of $\widetilde{K}(S^2)$. Let $\eta\in\widetilde{K}(S^3\wedge A)$ be the composite of the pinch map $S^3\wedge A\to S^{26}$ and a generator of $\pi_{26}(B\U(\infty))\cong\Z_{(5)}$. Then by \cite{HKMO}, we have
$$\ch(\xi)=\frac{5}{8!}u_{18}+\frac{5}{32\cdot 11!}u_{26},\quad\ch(\eta)=u_{26}.$$
Thus we get $\mathrm{Im}\,\Phi\circ\delta_*$ as desired. As in \cite{HKMO}, we can calculate $\Phi(\widetilde{\theta})$ by using the Chern classes of $\rho$ to get the desired result. The case of $(G,p,i)=(\E_8,11,9)$ is proved similarly.
\end{proof}

\begin{corollary}
\label{quasi_regular-1}
We have $\gamma_3(\E_6,5)=5^2$ and $\gamma_9(\E_8,11)=11^2$.
\end{corollary}

The following are proved similarly to Proposition \ref{quasi-regular}.

\begin{proposition}
For $(G,p,i)=(\E_6,7,5),(\E_7,11,7),(\E_7,13,5),(\E_8,13,11),(\E_8,19,11)$, there is an injection 
$$\Phi\colon[S^3\wedge A,F]\to H^{2i+2p+2}(S^3\wedge A;\Z_{(p)})\cong\Z_{(p)}$$
such that $\mathrm{Im}\,\Phi\circ\delta_*$ is generated by $1\in\Z_{(p)}$, where $A=A(2i+1,2i+1+2(p-1))$. Then we have $[S^3\wedge A,F]=0$.
\end{proposition}

\begin{corollary}
\label{quasi_regular-2}
$\gamma_5(\E_6,7)=\gamma_7(\E_7,11)=\gamma_5(\E_7,13)=\gamma_{11}(\E_8,13)=\gamma_{11}(\E_8,19)=1$.
\end{corollary}

\begin{proposition}
For $(G,p,i)=(\E_8,13,5)$, there is an injection 
$$\Phi\colon[S^3\wedge A(35,59),F]\to H^{62}(S^3\wedge A(35,59);\Z_{(13)})\cong\Z_{(13)}$$
such that $\mathrm{Im}\,\Phi\circ\delta_*$ is generated by $13\in\Z_{(13)}$ and a lift $\tilde{\theta}$ of $\langle\epsilon,\lambda_5\rangle$ can be chosen such that $\Phi(\tilde{\theta})=2^{-2}\cdot 3\cdot 5\cdot 2207\cdot 17977\in\Z_{(13)}$.
\end{proposition}

\begin{corollary}
\label{quasi_regular-3}
$\gamma_5(\E_8,13)=13$.
\end{corollary}

\begin{proposition}
For $(G,p,i)=(\E_8,19,5)$, there is an injection 
$$\Phi\colon\pi_{50}(F)\to H^{50}(S^{50};\Z_{(19)})\cong\Z_{(19)}$$
such that $\mathrm{Im}\,\Phi\circ\delta_*$ is generated by $19\in\Z_{(19)}$ and a lift $\tilde{\theta}$ of $\langle\epsilon,\lambda_i\rangle$ can be chosen such that $\Phi(\tilde{\theta})=-2\cdot 3^2\cdot 5\cdot 11\cdot 2861$.
\end{proposition}

\begin{corollary}
\label{quasi_regular-4}
$\gamma_5(\E_8,19)=19$.
\end{corollary}


\begin{proposition}
\label{quasi_regular-5}
$\gamma_2(\E_6,7)=7$ and $\gamma_3(\E_7,11)=11$.
\end{proposition}

\begin{proof}
We can prove this proposition in the same manner as above, but it needs more data from the Chern characters than that obtained in \cite{HKMO}. Therefore we employ another method also used in \cite{HKMO}. 
Consider the case $(G,p)=(\E_6,7)$. Recall that the mod 7 cohomology of $B\E_6$ is given by $\Z/7[x_4,x_{10},x_{12},x_{16},x_{18},x_{24}]$ for $|x_i|=i$. Let $\bar{\epsilon}\colon S^4\to B\E_6$ and $\bar{\lambda}_2\colon S^{18}\to B\E_6$ be the adjoints of $\epsilon\colon S^3\to\E_6$ and $\lambda_2\colon S^{17}\to\E_6$, respectively. Since $\bar{\epsilon}$ is the inclusion of the bottom cell and $\lambda_2$ is the inclusion of the direct product factor, we have $\bar{\epsilon}^*(x_4)=u_4$ and $\bar{\lambda}_2^*(x_{18})=u_{18}$ for a generator $u_i$ of $H^i(S^i;\Z/7)$. By \cite{HKMO}, we have that $\mathcal{P}^1x_{10}$ has no linear term and includes $x_4x_{18}$. Suppose that the Samelson product $\langle\epsilon,\lambda_2\rangle$ is trivial. Then, taking adjoints, the Whitehead product $[\bar{\epsilon},\bar{\lambda}_2]$ is trivial too, so there is a homotopy commutative diagram
$$\xymatrix{S^4\vee S^{18}\ar[r]^(.58){\bar{\epsilon}\vee\bar{\lambda}_2}\ar[d]^{\rm incl}&B\E_6\ar@{=}[d]\\
S^4\times S^{18}\ar[r]^(.58)\mu&B\E_6.}$$ 
Let $u_i$ be a generator of $H^i(S^i;\Z/7)$. It follows that $\mu^*(x_4)=u_4\otimes 1$ and $\mu^*(x_{18})=u_{18}$, so $$0=\mathcal{P}^1\mu^*(x_{10})=\mu^*(\mathcal{P}^1x_{10})=\mu^*(x_4x_{18})=u_4\otimes u_{18}\ne 0,$$ a contradiction. Thus the Samelson product $\langle\epsilon,\lambda_2\rangle$ is non-trivial. On the other hand, by Theorems \ref{homotopy group of sphere} and \ref{homotopy group of Bi}, $\langle\epsilon,\lambda_2\rangle$ factors through $S^9\subset\E_6$ and $7\cdot\pi_{20}(S^9)_{(11)}=0$. Therefore we obtain that $\gamma_2(\E_6,7)=7$ as desired. The case of $\gamma_3(\E_7,11)$ is similarly verified using the fact from~\cite{HKMO} that $\mathcal{P}^2x_{12}$ has no linear term and includes $-3x_4x_{28}$, and the fact that the mod 11 cohomology of $B\E_7$ is $\Z/11[x_4,x_{12},x_{16},x_{20},x_{24},x_{28},x_{36}]$ for $|x_i|=i$. 
\end{proof}

%
%
%

Since $\F_4$ is a retract of $\E_6$, we can deduce the following from Propositions \ref{quasi_regular-1} and
\ref{quasi_regular-2}. 

\begin{corollary}
\label{F_4}
Possible non-trivial $\langle\epsilon,\lambda_i\rangle$ in $\F_4$ at the prime $p$ are the cases $(p,i)=(5,2),(7,5)$ and $\gamma_3(\F_4,5)=5^2,\gamma_5(\F_4,7)=1$.
\end{corollary}


\section{The case $(G,p)=(\E_7,7),(\E_8,7)$}

We first consider the case $(G,p)=(\E_7,7)$, where $\E_7\simeq B(3,15,27)\times B(11,23,35)\times S^{19}$. As in \cite{MNT}, the H-spaces $B(3,15,27)$ and $B(11,23,35)$ are direct product factors of the mod~$p$ decomposition of $\SU(18)$. So we can determine their low dimensional homotopy groups from those of $\SU(18)$. Therefore Theorems \ref{homotopy group of sphere} and \ref{homotopy group of Bi} imply that $\langle\epsilon,\lambda_3\rangle$ is null homotopic. Thus we get:

\begin{proposition}
\label{E_7-1}
$\gamma_3(\E_7,7)=1$.
\end{proposition}

Let $A(11,23,35)$ be a homology generating subspace of $B(11,23,35)$.

\begin{proposition}
\label{E_7}
For $(G,p)=(\E_7,7)$, there is an injection 
$$\Phi\colon[S^3\wedge A(11,23,35),F]\to H^{38}(S^3\wedge A(11,23,35);\Z_{(7)})\cong\Z_{(7)}$$
such that $\mathrm{Im}\,\Phi\circ\delta_*$ is generated by $1$, implying that $[S^3\wedge A(11,23,35),F]=0$.
\end{proposition}

\begin{proof}
The proof is similar to that for Proposition \ref{quasi-regular} once we calculate $\mathrm{Im}\,\Phi\circ\delta_*$. 
Let $\bar{\rho}\colon\Sigma A(11,23,35)\to B\SU(\infty)$ be the adjoint of $\rho\circ\lambda_5\colon A(11,23,35)\to\SU(\infty)$ and put $\xi=\beta\wedge\bar{\rho}\in\widetilde{K}(S^3\wedge A(11,23,35))$ for a generator $\beta\in\widetilde{K}(S^2)$. Then by \eqref{HKMO} we have
$$\ch(\xi)=-\frac{1}{5!}u_{14}-\frac{9}{2\cdot 11!}u_{26}-\frac{1229}{60\cdot 17!}u_{38}$$
where $u_i$ is a generator of $H^i(S^3\wedge A(11,23,35);\Z_{(7)})\cong\Z_{(7)}$ for $i=14,26,38$. 
Thus we get $\mathrm{Im}\,\Phi\circ\delta_*$ as desired.
\end{proof}

\begin{corollary}
\label{E_7-2}
$\gamma_5(\E_7,7)=1$.
\end{corollary}

\begin{proposition}
\label{E_7-3}
$\gamma_1(\E_7,7)=7$.
\end{proposition}

\begin{proof}
By Theorems \ref{homotopy group of sphere} and \ref{homotopy group of Bi}, the Samelson product $\langle\epsilon,\lambda_1\rangle$ in $\E_7$ at $p=7$ factors through $S^{19}\subset\E_7$, so it is the composite of the pinch map onto the top cell $S^3\wedge A(3,15,27)\to S^{30}$ and an element of $\pi_{30}(S^{19})_{(7)}$. As $7\cdot\pi_{30}(S^{19})_{(7)}=0$, we get $\gamma_1(\E_7,7)\le 7$. 

Next, we show that $\langle\epsilon,\lambda_1\rangle$ is non-trivial. Recall that the mod 7 cohomology of $B\E_7$ is  
$$H^*(B\E_7;\Z/7)=\Z/7[x_4,x_{12},x_{16},x_{20},x_{24},x_{28},x_{36}],\quad|x_i|=i.$$
If we show that $\mathcal{P}^1x_{20}$ includes $\alpha x_4x_{28}$ for $\alpha\ne 0$ and does not include the linear term then we may proceed as in the proof of Proposition \ref{quasi_regular-5}. Let $j\colon\Spin(11)\to\E_7$ be the natural inclusion. In \cite{HKMO}, it is shown that the generators $x_i$ ($i=4,12,20,28$) are chosen as
$$j^*(x_4)=p_1,\quad j^*(x_{12})=p_3+p_2p_1,\quad j^*(x_{20})=p_5,\quad j^*(x_{28})=-2p_5p_2+\text{other terms,}$$
where $H^*(B\Spin(11);\Z/7)=\Z/7[p_1,p_2,p_3,p_4,p_5]$ and $p_i$ is the $i$-th Pontrjagin class. By  the mod $p$ Wu formula in \cite{S}, we can see that 
$$\mathcal{P}^1p_5=-p_5p_3+p_5p_2p_1+\text{other terms},$$
so for degree reasons, we get that $\mathcal{P}^1x_{20}$ includes $-x_4x_{28}$ and does not include the linear term.
\end{proof}

We next consider the case $(G,p)=(\E_8,7)$. As in the case of $(G,p)=(\E_7,7)$, we see that $\langle\epsilon,\lambda_1\rangle$ is trivial. Let $A$ be a homology generating subspace of $B(23,35,47,59)\subset\E_8$. 

\begin{proposition}
For $(G,p)=(\E_8,7)$, there is an injection 
$$\Phi\colon[S^3\wedge A,F]\to H^{50}(S^3\wedge A;\Z_{(7)})\oplus H^{62}(S^3\wedge A;\Z_{(7)})\cong(\Z_{(7)})^2$$
such that $\mathrm{Im}\,\Phi\circ\delta_*$ is generated by $(-2^2\cdot 3\cdot 7\cdot 19\cdot 199,5^3\cdot 7\cdot 13\cdot 31\cdot 61)$ and $(0,7^2)$, and a lift~$\tilde{\theta}$ of $\langle\epsilon,\lambda_5\rangle$ can be chosen such that $\Phi(\tilde{\theta})=(-2\cdot 3^2\cdot 5\cdot 11\cdot 2861,2^{-2}\cdot 3\cdot 5\cdot 2207\cdot 17977)$.
\end{proposition}

\begin{proof}
Let $\xi_1=\beta\wedge\bar{\rho}\in\widetilde{K}(S^3\wedge A)$ be as in the proof of Proposition \ref{E_7}. Let 
$$\xi_2=\frac{1}{2^{13}\cdot 3^2}(\phi^2(\xi_1)-2^{13}\xi_1),\quad\xi_3=\frac{1}{2^{19}\cdot 3^2\cdot 13}(\phi^2(\xi_2)-2^{19}\xi_2)$$ 
and let $\xi_4$ be the composite of the pinch map $S^3\wedge A\to S^{62}$ and a generator of $\pi_{62}(B\U(\infty))$. Put $a_1=2\cdot 3^2\cdot 7,a_2=2^2\cdot 5\cdot 7\cdot 13,a_3=2\cdot 3\cdot 7\cdot 11\cdot 19\cdot 199,a_4=2\cdot 5^3\cdot 7\cdot 11\cdot 13\cdot 61$. Then by \cite{HKMO}, we have:
\begin{align*}{}
\ch(\xi_1)&=\frac{a_1}{11!}u_{26}+\frac{a_2}{17!}u_{38}-\frac{a_3}{5\cdot 23!}u_{50}+\frac{a_4}{29!}u_{62}\\
\ch(\xi_2)&=\frac{7a_2}{17!}u_{38}-\frac{7\cdot 13a_3}{23!}u_{50}+\frac{3\cdot 7\cdot 19\cdot 73a_4}{29!}u_{62}\\
\ch(\xi_3)&=-\frac{7^2a_3}{23!}u_{50}+\frac{3\cdot 5\cdot 7^2\cdot 19\cdot 73a_4}{29!}u_{62}\\
\ch(\xi_4)&=u_{62}
\end{align*}
where $u_i$ is a generator of $H^i(S^3\wedge A;\Z_{(7)})\cong\Z_{(7)}$ for $i=26,38,50,62$. Then $\mathrm{Im}\,\Phi\circ\delta_*$ can be calculated as asserted, and $\Phi(\tilde{\theta})$ can be calculated as in \cite{HKMO}.
\end{proof}

\begin{corollary}
\label{E_8}
We have $\gamma_1(\E_8,7)=1$ and $\gamma_5(\E_8,7)=7^2$.
\end{corollary}

\begin{proof}
[Proof of Theorem \ref{main}]
By Propositions \ref{max}, \ref{quasi_regular-5}, \ref{E_7-1}, \ref{E_7-3} and Corollaries \ref{cor:regular}, \ref{quasi_regular-1}, \ref{quasi_regular-2}, \ref{quasi_regular-3}, \ref{quasi_regular-4}, \ref{F_4}, \ref{E_7-2}, \ref{E_8}, we obtain the following table:
\renewcommand{\arraystretch}{1.2}
\begin{table}[H] 
\centering
\begin{tabular}{l|lllllllllll} 
$p$&5&7&11&13&17&19&23&29&31&$\ge 37$\\\hline
$\gamma(\F_4,p)$&$5^2$&1&1&13&1&1&1&1&1&1\\
$\gamma(\E_6,p)$&$5^2$&7&1&13&1&1&1&1&1&1\\
$\gamma(\E_7,p)$&&7&11&1&1&19&1&1&1&1\\
$\gamma(\E_8,p)$&&$7^2$&$11^2$&13&1&19&1&1&31&1
\end{tabular}
\end{table} 
\noindent 
Therefore $\gamma(G)$ is the product of $\gamma(G,p)$ for $p\ge 5$ ($G=\F_4,\E_6$) and $p\ge 7$ ($G=\E_7,\E_8$), so the assertion of Theorem \ref{main} follows from Theorem \ref{KKT}. The $\G_2$ case was proved in \cite{KTT}.
\end{proof}

\end{document}